\documentclass{article}
\newcommand{\paperTitle}{On the Duality between Feature and Sample Screening}

\newcommand{\paperAbstract}{
Feature and sample screening reduce the cost of machine learning by eliminating irrelevant features and noninformative samples, respectively. Although recognized as primal--dual counterparts, their relationship remains informal and model-dependent. Viewing screening and duality as transformations of objective functions, we introduce Fenchel--Rockafellar (FR) representations, a class of convex problems encompassing the Lasso and SVM that is closed under both transformations. We then prove that feature and sample screening form an equivariant pair: dualization followed by feature screening is equal to sample screening followed by dualization.
}

\newcommand{\paperKeywords}{Feature screening, Sample screening, Lasso, SVM, Fenchel--Rockafellar duality}

\newcommand{\authorLe}{Thu-Le Tran}
\newcommand{\affilLe}{Can Tho University, Vietnam}
\newcommand{\emailLe}{ttle@ctu.edu.vn}

\usepackage{xcolor}

\newcommand{\noteCLAUDE}[1]{}
\newcommand{\addCLAUDE}[1]{}
\usepackage{authblk}

\usepackage[left=3cm, right=3cm]{geometry}
\usepackage{amsmath, amssymb, amsthm}

\usepackage[colorlinks=true,
            linkcolor=blue,     % màu cho \ref, \eqref,...
            citecolor=red,      % màu cho \cite
            urlcolor=blue]{hyperref}   % luôn load trước
\usepackage[capitalise,nameinlink,noabbrev]{cleveref}

\newtheorem{theorem}{Theorem}[section]

\newtheorem{proposition}[theorem]{Proposition}
\newtheorem{corollary}[theorem]{Corollary}

\theoremstyle{definition}
\newtheorem{definition}[theorem]{Definition}

\newtheorem{example}[theorem]{Example}

\crefname{lemma}{Lemma}{Lemmas}
\Crefname{lemma}{Lemma}{Lemmas}
\crefname{proposition}{Proposition}{Propositions}
\Crefname{proposition}{Proposition}{Propositions}
\crefname{corollary}{Corollary}{Corollaries}
\Crefname{corollary}{Corollary}{Corollaries}
\crefname{consequence}{Consequence}{Consequences}
\Crefname{consequence}{Consequence}{Consequences}
\crefname{definition}{Definition}{Definitions}
\Crefname{definition}{Definition}{Definitions}
\crefname{assumption}{Assumption}{Assumptions}
\Crefname{assumption}{Assumption}{Assumptions}
\crefname{remark}{Remark}{Remarks}
\Crefname{remark}{Remark}{Remarks}
\crefname{example}{Example}{Examples}
\Crefname{example}{Example}{Examples}

\usepackage{tikz-cd}

\usepackage{float}

\usepackage[ruled,vlined,linesnumbered]{algorithm2e}
\SetKwInput{KwInput}{Input}
\SetKwInput{KwOutput}{Output}
\SetKwInput{KwInit}{Initialize}

\usepackage{pgfplots}
\pgfplotsset{compat=1.18}
\usetikzlibrary{calc}

\usepackage{amsmath}
\usepackage{xcolor}
\usepackage[square,authoryear]{natbib}
\title{\paperTitle}
\date{\today}
\author[1]{\authorLe}
\affil[1]{\affilLe \\ \texttt{\emailLe}}
\begin{document}
\maketitle
\begin{abstract}
    \paperAbstract
\end{abstract}


\tableofcontents

\section{Introduction}
\label{sec:intro}

% \noteCLAUDE{Context: large-scale and screening.}\noteLE{Verified.}
Large-scale convex optimization is central to modern machine learning. In recent years, \emph{screening} methods have emerged as an effective technique for reducing problem size before or during optimization, thereby accelerating the solution of large-scale learning problems \citep{elghaoui2010}.

% \noteCLAUDE{Classification of screening}
% \noteLE{Verified.}
Screening methods have been studied from several complementary perspectives. According to whether they guarantee preservation of the optimal solution, they are broadly classified into safe screening methods \citep{elghaoui2010} and unsafe screening methods \citep{tibshirani2012strong,fan2018sure}. According to how screening is
performed, they may be static
\citep{elghaoui2010}, dynamic
\citep{bonnefoy2014dynamic}, or sequential
\citep{wang2013lasso}. Recent advances mainly focus on constructing
tighter safe regions, including ball
\citep{ndiaye2017gap,tran2025one}, dome
\citep{tran2022beyond}, and ellipsoid
\citep{dai2012ellipsoid} regions. The screening paradigm has also been generalized to related reduction mechanisms, including squeezing \citep{elvira2020safe}, relaxing \citep{guyard2022screen}, and peeling \citep{guyard2023safe}.

% \noteCLAUDE{SCOPE}\noteLE{Verified.}
This paper is concerned with a different axis of classification,
namely the screening target: \emph{feature screening} and \emph{sample screening}.
% Note that rather than studying the safe screening, here, we focus on general screening framework. 

% \noteCLAUDE{Historical development.}\noteLE{Verified.}
Feature screening was introduced for sparse learning models such as the
Lasso, where inactive features are identified and removed
\citep{elghaoui2010}. Sample screening was later developed by transferring
this idea to the dual of the soft-margin SVM, whose dual variables
correspond to training samples \citep{ogawa2013}. This primal--dual
connection also appears in simultaneous screening
\citep{shibagaki2016simultaneous} and in frameworks based on
Fenchel--Rockafellar duality
\citep{ndiaye2017gap,yamada2021dynamic,tran2025one}.

Across these developments, various existing general screening frameworks focus
primarily on deriving screening rules or safe regions, rather than on
formalizing the mathematical relationship between feature and sample screening. To the best of our knowledge, such a formal relationship is still missing. We therefore pose the following question, which is the main motivation of this paper:

\begin{center}
\emph{In what precise mathematical sense are feature screening and sample screening dual?}
\end{center}

In this paper, we address this question at the level of general screening, independently of the safety conditions required in safe screening.

To formalize their relationship, we view screening and FR duality as transformations on a common class of objective functions. Specifically,
we consider problems of the form
\[
    \min_{x\in\mathbb R^n} p(x)=f(Ax)+g(x)+a,
    \qquad A\in\mathbb R^{m\times n},\quad a\in\mathbb R,
\]
where $f$ and $g$ are \emph{separable} proper closed convex functions. We encode and identify the objective function $p$ with the quadruple $p=(f,g,A,a)$, called an \emph{FR representation} of $p$. We 
denote the class of all such representations by $\Gamma$.
% This class is
% closed under both screening and FR duality.
The scalar $a$ is included to
preserve closure under the transformations considered in this paper. The separability assumption of $f$ and $g$ is standard in machine learning and is essential for screening.

This class is particularly suited to supervised learning models: the columns of
$A$ represent features, while its rows represent samples. It encompasses
regression models such as the Lasso \citep{elghaoui2010}, Elastic Net
\citep{guyard2022screen}, and Huber regression \citep{chen2020safe};
classification models such as logistic regression
\citep{wang2014safe} and SVMs \citep{ogawa2013,nguyen2026gap}; and
optimal transport \citep{su2024safe}.

The corresponding FR dual problem, modified to account for the constant
$a$, is \citep{rockafellar2015convex}
\[
\max_{y \in \mathbb R^m}\;
-p^*(y), \qquad p^*(y) := g^*(-A^\top y) + f^*(y) - a.
\]
Here, \(f^*\) and \(g^*\) are the Fenchel conjugates of \(f\) and \(g\), while \(p^*\) is the FR dual of \(p\).

% Throughout the introduction, we write $S^F$ and $S^S$ for the feature- and sample-screening transformations, respectively, suppressing their
% index and certificate parameters for readability.

Feature screening, restated and generalized in this paper, assumes
that a block of variables satisfies \(x_{\bar I}=z\), where
\(I\sqcup\bar I=\{1,\ldots,n\}\), and transforms the function \(p\)
into\footnote{Throughout the introduction, we omit the subscript
parameters of \(S^F\) and \(S^S\) for readability.}
\[
    S^F(p)
    =
    f(A_Ix_I+A_{\bar I}z)
    +g_I(x_I)
    +a
    +g_{\bar I}(z).
\]
Classical feature screening is recovered by the special case
\(z=0\), see e.g., \citep{elghaoui2010}.

Sample screening, restated and generalized in this paper, assumes
\(s\in\partial f_{\bar J}(A_{\bar J}x)\), where
\(J\sqcup\bar J=\{1,\ldots,m\}\). By the Fenchel--Young equality, the loss function \(f\) at block \(\bar J\) can be replaced by an affine function,
\(
f_{\bar J}(A_{\bar J}x)
=
\langle s,A_{\bar J}x\rangle
-
f_{\bar J}^*(s).
\)
Substituting this into the objective, the function \(p\) becomes
\[
    S^S(p)
    =
    f_J(A_Jx)
    +g(x)
    +\langle A_{\bar J}^{\top}s,x\rangle
    +a
    -f_{\bar J}^*(s).
\]
Classical sample screening for the soft-margin SVM \citep{ogawa2013} is recovered by the special case
\(s=0\).

% \noteCLAUDE{Answer to Q2.}\noteLE{Verified.}
Now \(S^F\), \(S^S\), and \((\cdot)^*\) can be considered as transformations on
\(\Gamma\). 
This allows us to answer the question posed above: feature screening and sample screening are dual in the sense that they form an \emph{equivariant pair} under FR duality.
\[
    (S^F(p))^*
    =
    S^S(p^*).
\]
Specifically, feature screening on a primal representation is exactly sample screening on its dual representation. The duality between the
two screening transformations is expressed by the
commutative diagram
\[
\begin{array}{ccc}
p & \xrightarrow{S^F} & S^F(p)\\
{\scriptstyle *}\big\downarrow & & \big\downarrow{\scriptstyle *}\\
p^* & \xrightarrow{S^S} & S^S(p^*)
\end{array}
\]

% \noteCLAUDE{Contributions.}\noteLE{Verified.}
To summarize, the contributions of this paper are twofold. First, we formulate feature
screening and sample screening as transformations on the class of FR
representations. Second, we provide a precise mathematical formulation
of the duality between feature screening and sample screening via
equivariance under FR duality.

To achieve this, our approach is to show that both screening transformations
decompose into a translation followed by a restriction, and that these
primitive transformations are equivariant under FR duality. Within this
approach, reduction in problem size and the equivariance of
simultaneous screening follow as immediate consequences.

% \noteCLAUDE{Organization.}\noteLE{Verified.}
The rest of the paper is organized as follows. Section~\ref{sec:operators} introduces FR representations together with the primitive transformations including FR duality, translation, and restriction acting on them, providing the common mathematical environment for screening. Section~\ref{sec:screening} then defines feature screening and sample screening as transformations on this class of representations, proves that both decompose into a translation followed by a restriction, and establishes the duality between feature and sample screening.

\section{FR Representations and Equivariant Transformations}
\label{sec:operators}
This section does not study screening directly; it introduces FR representations together with the three primitive transformations that act on them, FR duality, translation, and restriction, and establishes that translation and restriction are equivariant under FR duality.

Throughout this section, functions take values in the extended reals \(\overline{\mathbb R}:=\mathbb R\cup\{+\infty\}\), so that a domain constraint can be encoded as \(+\infty\) outside the feasible set. For a finite set \(N\) and \(I\subseteq N\), \(\bar I:=N\setminus I\) denotes its complement, and  disjoint-union \(I\sqcup\bar I=N\) records that \(I\) and \(\bar I\) partition \(N\). This disjoint-union notation is used for both index sets fixed below: \(I\sqcup\bar I=\{1,\ldots,n\}\) and \(J\sqcup\bar J=\{1,\ldots,m\}\).

%%%%%%%%%%%%%%%%%%%%%%%%%%%%%%%%%%%%%%%%%%%%%%%%%%%%%%%%%%%%%%%

\subsection{FR Representations}

We consider the problems of the form
\begin{equation}
\min_{x \in \mathbb R^n}\;
p(x), \qquad p(x) = f(Ax)+g(x)+a.
\label{eq:FR-problem}
\end{equation}
where
\begin{equation}
f:\mathbb R^m\rightarrow\overline{\mathbb R},
\qquad
g:\mathbb R^n\rightarrow\overline{\mathbb R},
\qquad
A\in\mathbb R^{m\times n},
\qquad
a\in\mathbb R,
\label{eq:cond-of-FR-representation}
\end{equation}
with \(f\) and \(g\) closed, proper, and convex.
Here, at first glance, \(a\) is a redundant constant since it does not affect the minimizer of \(p\). However, it is important to retain the constant \(a\) because it is used as an absorbing constant in the following transformations.

% \noteLE{add ref}.
The problem in \eqref{eq:FR-problem} is fundamental to convex optimization and machine learning. It is the standard form underlying FR duality and provides a unified formulation for various screening frameworks \citep{tran2025one}.

\begin{definition}[FR representation]
    If \((f,g,A,a)\) satisfies \eqref{eq:cond-of-FR-representation}, then we say that it is admissible. In this case, we refer to it as an FR representation of function \(p\).
\end{definition}

In this paper, we identify the optimization problem \eqref{eq:FR-problem} with its FR representation \((f,g,A,a)\): the quadruple is not just a convenient way to write down \(p\), but is designed so that the class of FR representations stays \emph{closed} under the transformations studied in this paper.

For the function \(p\) in \eqref{eq:FR-problem}, we write \(\operatorname{size}(p):=(m,n)\) for the \emph{problem size} of $p$.
We define \(\Gamma_{m, n}\) to be the set of all functions \(p\) with problem size \((m, n)\), and \(\Gamma\) to be the set of all FR representations of any size.
If both \(f\) and \(g\) are separable, then \((f,g,A,a)\) is said to be a \emph{separable FR representation}. %Raw complexity and separability are essential and will be discussed in more detail below.

%%%%%%%%%%%%%%%%%%%%%%%%%%%%%%%%%%%%%%%%%%%%%%%%%%%%%%%%%%%%%%%

\subsection{FR Duality and Equivariant Transformations}

% main message: the dual of an FR representation is another FR representation, and duality is an involution on the class of FR representations.

We now recall FR duality and show that it defines a closed transformation on $\Gamma$.

For a function \(\varphi:\mathbb R^\ell\to\overline{\mathbb R}\), its Fenchel
conjugate is
\[
\varphi^*(y):=\sup_{x\in\mathbb R^\ell}\{\langle y,x\rangle-\varphi(x)\}.
\]
When \(\varphi\) is proper closed convex, then
\(\varphi^{**}=\varphi\) \citep{rockafellar2015convex}, a fact used repeatedly below.% \noteLE{add ref.}

The FR dual problem of \eqref{eq:FR-problem} is \citep{rockafellar2015convex}
\[
\max_{y \in \mathbb R^m}\;
-q(y), \qquad q(y) = g^*(-A^\top y) + f^*(y)-a.
\]
Note that in the classical FR duality, we do not have \(a\) and \(-a\) in \(p\) and \(q\).
Here the constant \(a\) is negated in the dual, so that the weak duality inequality \(p(x) + q(y)\geq 0\) remains valid for all \(x\in \mathbb R^n\) and \(y \in \mathbb R^m\).

Since the dual function \(q\) is itself an FR representation, with the problem size \((n, m)\), FR duality defines a transformation on the class of FR representations.

\begin{definition}[FR Duality]
The FR duality is a transformation \((\cdot)^*: \Gamma_{m, n} \to \Gamma_{n, m}\) such that the image of \(p=(f,g,A,a)\) is
\[
p^*
:=
(g^*,\,f^*,\,-A^\top,\,-a).
\]
\end{definition}

The meaning of ${}^*$ depends on the type of its argument: $\varphi^*$ denotes the Fenchel conjugate of a function $\varphi$, whereas $p^*$ denotes the FR dual of an FR representation $p$. We adopt this slight abuse to simplify notation.

% \noteCLAUDE{This result is used for simultaneous screening.} \noteLE{Verified.}
\begin{proposition}[FR duality involution]
\label{prop:involution}
For every FR representation \(p \in \Gamma\), we have \((p^*)^*=p\).
\end{proposition}

\begin{proof}
Write \(p=(f,g,A,a)\), so \(p^*=(g^*,f^*,-A^\top,-a)\). Applying the
same rule again,
\[
(p^*)^*
=
\bigl((f^*)^*,\,(g^*)^*,\,-(-A^\top)^\top,\,-(-a)\bigr)
=
(f^{**},g^{**},A,a)
=
(f,g,A,a)
=
p,
\]
using biconjugation \(f^{**}=f\), \(g^{**}=g\), which hold since \(f\) and \(g\) are closed, proper, and convex \citep{rockafellar2015convex}.
\end{proof}

Then FR duality naturally induces a notion of duality between transformations, called \emph{equivariance}; this is the main structure to investigate in the remainder of the paper.

\begin{definition}[Equivariance under FR duality]
\label{def:equivariance}
Transformations \(F,G:\Gamma\to\Gamma\) are \emph{equivariant under FR duality} if
\[
(F(p))^*=G(p^*),
\qquad
\forall\,p\in\Gamma.
\]
\end{definition}

%%%%%%%%%%%%%%%%%%%%%%%%%%%%%%%%%%%%%%%%%%%%%%%%%%%%%%%%%%%%%%%

\subsection{Translation}
\label{subsec:translation}
The second transformation is translation. We first introduce three
primitive translations that combine into it.

\begin{definition}[Primitive translations]
For \(b\in\mathbb R^n\), \(c\in\mathbb R\), and \(d\in\mathbb R^n\),
define, for a function \(\varphi: \mathbb R^n \rightarrow \overline{\mathbb R}\),
\[
\text{domain translation:}
\quad
(T^{\mathrm d}_b \varphi)(y):=\varphi(y+b),
\]
\[
\text{value translation:}
\quad
(T^{\mathrm v}_c \varphi)(y):=\varphi(y)+c,
\]
\[
\text{slope translation:}
\quad
(T^{\mathrm s}_d\varphi)(y):=\varphi(y)+\langle d,y\rangle.
\]
\end{definition}

Here, the superscript specifies the type of translation, while the subscript specifies its parameter.
Geometrically, domain translation shifts the graph horizontally,
value translation shifts it vertically, and slope translation tilts
the graph by adding a linear function, thereby shifting every
subgradient by the fixed vector \(d\).

\begin{definition}[Translation parameter and its space]
\label{def:translation-parameter}
We call
\(
K_{m,n}:=\mathbb R^m\times\mathbb R\times\mathbb R^n,
\)
a \emph{translation parameter space}.
An element \(k=(b,c,d)\in K_{m,n}\) is called a \emph{translation
parameter}.
\end{definition}

Now, we can define a translation indexed by a translation parameter.

\begin{definition}[Translation]
Let \(p=(f,g,A,a)\in \Gamma_{m, n}\). The translation \(T_k\) indexed by \(k=(b,c,d) \in K_{m, n}\) on \(p\) is defined by \(T_k : \Gamma_{m, n} \rightarrow \Gamma_{m, n} \) such that
\[
T_k(p)
:=
(T^{\mathrm d}_bf,\;
T^{\mathrm s}_dg,\;
A,\;
a+c).
\]
\end{definition}

Note that the matrix \(A\) is untouched by translation: only the two functions
and the constant move. %This is what lets restriction and translation later be composed without the matrix itself being disturbed by which primitive acts first.
It is clear that the family \(\{T_k\}\) forms an Abelian translation action on the class of FR representations:
\[
T_{k_1+k_2}(p)
=
T_{k_1}(T_{k_2}(p)) = T_{k_2}(T_{k_1}(p)). 
\]
% \noteCLAUDE{Not currently cited by any result} \noteLE{verified. But kept because of its beauty. So keep it of this form, do not replace it by a proposition. The commutativity will be used in the next paper.}

Section~\ref{sec:screening} will apply the translation action by
choosing \(k\) as a function of the coordinates being screened; we
first establish this subsection's law, how the translation action
interacts with FR duality.

\begin{proposition}[Primitive translation equivariance]
\label{prop:duality-of-primitive-translations}
For every proper closed convex function \(\varphi: \mathbb{R}^n \to \overline{\mathbb{R}}\), \(b,d \in \mathbb{R}^n\), and \(c \in \mathbb{R}\), we have
\begin{enumerate}
    \item \((T^{\mathrm d}_b \varphi)^* = T^{\mathrm s}_{-b}(\varphi^*)\)
    \item \((T^{\mathrm v}_c \varphi)^* = T^{\mathrm v}_{-c}(\varphi^*)\)
    \item \((T^{\mathrm s}_d \varphi)^* = T^{\mathrm d}_{-d}(\varphi^*)\)
\end{enumerate}
\end{proposition}
\begin{proof}
    The identities follow directly from the definition of the Fenchel conjugate.
\end{proof}
% \noteCLAUDE{Used in
% in the proof of Translation equivariance below.}\noteLE{Verified.}

We define the \emph{dual parameter} of \(k\) by
\(k^*=(-d,-c,-b)\in K_{n, m}\).
Intuitively, $k^*$ is obtained from $k$ by swapping the domain and slope translation parameters, and negating all three parameters. 
Note that $k^*\neq -k$ and \((k^*)^*=k\). 
Here, one recalls that the meaning of ${}^*$ is determined by the type of its argument.

\begin{proposition}[Translation equivariance]
\label{thm:fenchel-equivariance}
Let \(p\in \Gamma_{m, n}\) be an FR representation and let \(k=(b,c,d)\in K_{m, n}\) be a translation
parameter. Then \(T_k\) and \(T_{k^*}\) are equivariant under FR
duality:
\[
(T_k(p))^*
=
T_{k^*}(p^*).
\]
\end{proposition}

\begin{proof}
Write \(p=(f,g,A,a)\), so \(T_k(p)=(T^{\mathrm d}_bf,\,T^{\mathrm
s}_dg,\,A,\,a+c)\). Applying the FR dual and  Proposition~\ref{prop:duality-of-primitive-translations},
\[
(T_k(p))^*
=
\big(
(T^{\mathrm s}_dg)^*,\;
(T^{\mathrm d}_bf)^*,\;
-A^\top,\;
-(a+c)
\big)
=
\big(
T^{\mathrm d}_{-d}(g^*),\;
T^{\mathrm s}_{-b}(f^*),\;
-A^\top,\;
-a-c
\big).
\]

On the other side, we have \(p^*=(g^*,f^*,-A^\top,-a)\), and \(k^*=(-d,-c,-b)\),
so \(T_{k^*}(p^*)\). Thus, % applies a domain translation by \(-d\) to \(p^*\)'s first slot, a slope translation by \(-b\) to its second slot, and adds \(-c\) to its constant:
\[
T_{k^*}(p^*)
=
\big(
T^{\mathrm d}_{-d}(g^*),\;
T^{\mathrm s}_{-b}(f^*),\;
-A^\top,\;
-a-c
\big).
\]
Every term matches \((T_k(p))^*\), which proves the result.
\end{proof}

\noteCLAUDE{Used in  Section~\ref{sec:screening}'s Feature--sample
duality theorem.}
% \noteLE{Verified.}
%%%%%%%%%%%%%%%%%%%%%%%%%%%%%%%%%%%%%%%%%%%%%%%%%%%%%%%%%%%%%%%

%%%%%%%%%%%%%%%%%%%%%%%%%%%%%%%%%%%%%%%%%%%%%%%%%%%%%%%%%%%%%%%

\subsection{Restriction}

The third and final transformation is restriction. Throughout this
subsection and Section~\ref{sec:screening}, we additionally assume that \(f\) and \(g\) are separable, \(f(y)=\sum_{j=1}^m f_j(y_j)\)
and \(g(x)=\sum_{i=1}^n g_i(x_i)\), with each \(f_j\) and \(g_i\) closed, proper, and convex on \(\mathbb{R}\).

Let \(I\subseteq\{1,\ldots,n\}\) and \(J\subseteq\{1,\ldots,m\}\). Write
\(x_I:=(x_i)_{i\in I}\) and \(g_I(x_I):=\sum_{i\in I}g_i(x_i)\); define
\(f_J\) symmetrically.
For a matrix \(A\in\mathbb R^{m\times n}\), write \(A_I:=A_{:,I}\) for
the submatrix of columns indexed by \(I\), and \(A_J:=A_{J,:}\) for the
submatrix of rows indexed by \(J\). Combining both,
\(A_{J,I}:=(A_J)_I=(A_I)_J\) denotes the submatrix with rows in \(J\)
and columns in \(I\).

\begin{definition}[Feature and sample restriction]
Let \(I\sqcup\bar I=\{1,\ldots,n\}\) and \(J\sqcup\bar J=\{1,\ldots,m\}\). The feature restriction operator \(R^{F}_I:\Gamma_{m,n}\to\Gamma_{m,|I|}\) is defined by
\[
R^{F}_I(f,g,A,a)
:=
(f,g_I,A_I,a).
\]
Symmetrically, the sample restriction operator \(R^{S}_J:\Gamma_{m,n}\to\Gamma_{|J|,n}\) is defined by
\[
R^{S}_J(f,g,A,a)
:=
(f_J,g,A_J,a).
\]
\end{definition}

The two restriction operators commute, \(R^{S}_J\bigl(R^{F}_I(p)\bigr)=R^{F}_I\bigl(R^{S}_J(p)\bigr)\):
feature restriction touches only \(g\) and the columns of \(A\), sample
restriction touches only \(f\) and the rows of \(A\), so the two act on
disjoint parts of \(p\). That same separation is what lets them
exchange cleanly under FR duality.

\begin{proposition}[Restriction equivariance]
\label{prop:restriction-equivariance}
For every \(p=(f,g,A,a)\in\Gamma\), every \(I\subseteq\{1,\ldots,n\}\),
and every \(J\subseteq\{1,\ldots,m\}\), \(R^F_I\) and \(R^S_I\) are
equivariant under FR duality (Definition~\ref{def:equivariance}), and
symmetrically for \(R^S_J\) and \(R^F_J\):
\[
\bigl(R^{F}_I(p)\bigr)^*=R^{S}_I(p^*),
\qquad
\bigl(R^{S}_J(p)\bigr)^*=R^{F}_J(p^*).
\]
\end{proposition}

\begin{proof}
Write \(p^*=(g^*,f^*,-A^\top,-a)\). Applying the FR dual termwise to
\(R^F_I(p)=(f,g_I,A_I,a)\),
\[
\bigl(R^F_I(p)\bigr)^*
=
\bigl((g_I)^*,\;f^*,\;-A_I^\top,\;-a\bigr).
\]
By separability of \(g\), \((g_I)^*=(g^*)_I\). Rows \(I\) of
\(-A^\top\) equal \(-A_I^\top\), since rows of \(A^\top\) are columns of
\(A\) transposed. Hence
\[
\bigl(R^F_I(p)\bigr)^*
=
\bigl((g^*)_I,\;f^*,\;-A_I^\top,\;-a\bigr)
=
R^S_I(p^*),
\]
the last equality by the definition of sample restriction applied to
\(p^*=(g^*,f^*,-A^\top,-a)\). The second identity is the mirror
computation, restricting \(R^S_J(p)=(f_J,g,A_J,a)\) instead, and using
\((f_J)^*=(f^*)_J\) by the same separability argument.
\end{proof}

% \noteCLAUDE{Restriction equivariance is used in the proof of the Feature--sample duality theorem in Section~\ref{sec:screening}.} \noteLE{Verified.}

With all three transformations in place, each carrying its own law of
equivariance under FR duality, Section~\ref{sec:screening} turns to
screening itself: not a fourth primitive, but a transformation built by
combining translation and restriction.

\section{Duality of Feature and Sample Screening}
\label{sec:screening}

This section shows that feature screening and sample screening form an
equivariant pair under FR duality. We define both as transformations on
FR representations, show that each decomposes into a translation
followed by a restriction, and use this decomposition to prove the
main duality identity. As a by-product, we obtain a reduction in
problem size and a characterization of simultaneous screening.

\subsection{Feature and Sample Screening as Transformations}

This subsection defines feature screening and sample screening as
transformations on \(\Gamma\). Recall that an FR
representation \(p=(f,g,A,a)\) encodes the objective function
\[
p(x)=f(Ax)+g(x)+a.
\]

In the literature, feature screening typically eliminates a block of primal variables. We state the
assumption slightly more generally, allowing the eliminated block to
sit at an arbitrary fixed point rather than requiring it to be zero.
Let \(I\sqcup\bar I=\{1,\ldots,n\}\) and \(z\in\mathbb R^{\bar I}\) be a known vector, and
suppose that
\[
x_{\bar I}=z.
\]
Writing \(x=(x_I,x_{\bar I})\), \(A=(A_I,A_{\bar I})\), and
\(g(x)=g_I(x_I)+g_{\bar I}(x_{\bar I})\) by separability, substituting
\(x_{\bar I}=z\) into the objective gives
\[
p(x_I,z)
=
f(A_Ix_I+A_{\bar I}z)+g_I(x_I)+a+g_{\bar I}(z).
\]

\begin{definition}[Feature screening]
The \emph{feature screening} of \(p\) at \((I,z)\) is the FR
representation
\[
S^{F}_{I,z}(p)
:=
\bigl(
T^{\mathrm d}_{A_{\bar I}z}f,\;
g_I,\;
A_I,\;
a+g_{\bar I}(z)
\bigr).
\]
\end{definition}

Here, \(S^F_{I,z}(p)(x_I)=p(x_I,z)\). This means the screened objective is exactly the original objective with \(x_{\bar I}\) fixed at \(z\).

\begin{example}[Feature screening for the Lasso]
\label{ex:lasso-feature}
Classical feature screening for the Lasso is the special case \(z=0\)
of the proposed feature screening operator \(S^F_{I,z}\). The Lasso
problem
\[
\min_{x\in\mathbb R^n}
\frac12\|Ax-b\|_2^2+\lambda\|x\|_1,
\qquad \lambda>0,
\]
is an FR representation \(p=(f,g,A,0)\) with
\(f(y)=\frac12\|y-b\|_2^2\) and \(g(x)=\lambda\|x\|_1\).

In sparse learning, feature screening certifies that a block of
coefficients is inactive at the optimum. For the Lasso, inactivity means
that the corresponding coefficients vanish,
\[
x_{\bar I}=0,
\]
exactly the special case \(z=0\) of the general feature-screening
assumption \(x_{\bar I}=z\). Since \(z=0\), the screened function
\(S^F_{I,0}(p)\) is simply the restriction of the Lasso representation
to the remaining coordinates, with corresponding problem
\[
\min_{x_I\in\mathbb R^{|I|}}
\frac12\|A_Ix_I-b\|_2^2+\lambda\|x_I\|_1.
\]
The proposed definition thus recovers classical Lasso feature
screening, while extending it from coefficients certified to be zero
to coefficients certified to take any fixed value \(z\).
\end{example}

Intuitively, feature screening simplifies the objective function by fixing a block of the variables in \(x\). Sample screening simplifies the objective function by replacing a block of the loss function \(f\) by its affine part. We now provide a formal definition of sample screening.

Let
\(J\sqcup\bar J=\{1,\ldots,m\}\) and \(s\in\mathbb R^{\bar J}\), and
suppose that
\begin{equation}
    \label{eq:general-sample-screening-condition}
    s\in\partial f_{\bar J}(A_{\bar J}x),
\end{equation}
i.e., \(s\) is a fixed subgradient of the block \(\bar J\) of \(f\) evaluated at
\(x\). By the Fenchel--Young inequality \(f(y)+f^*(s)\geq\langle s,y\rangle\), which holds with equality exactly when \(s\in\partial
f(y)\), the corresponding block of \(f\) can be replaced by a simple affine function, i.e.,
\[
f_{\bar J}(A_{\bar J}x)
=
\langle s,A_{\bar J}x\rangle-f_{\bar J}^{*}(s).
\]
Substituting this into the objective gives
\[
p(x)
=
f_J(A_Jx)+g(x)+\langle A_{\bar J}^{\top}s,x\rangle+a-f_{\bar J}^{*}(s).
\]

The above analysis motivates the definition of sample screening, independent of the condition \eqref{eq:general-sample-screening-condition}. 

\begin{definition}[Sample screening]
The \emph{sample screening} of \(p\) at \((J,s)\) is 
\[
S^{S}_{J,s}(p)
:=
\bigl(
f_J,\;
T^{\mathrm s}_{A_{\bar J}^{\top}s}g,\;
A_J,\;
a-f_{\bar J}^{*}(s)
\bigr),
\]
\end{definition}

For $x \in \mathbb{R}^n$ and $s\in \mathbb{R}^m$, we have $$S^S_{J,s}(p)(x) \leq  p(x),$$ i.e., the screened objective is less than or equal to the original objective and the equality holds if and only if the condition \eqref{eq:general-sample-screening-condition} is satisfied. 

Geometrically, sample screening replaces the eliminated loss block by
its Fenchel--Young affine representation. This
replacement potentially reduces the computational cost of the resulting problem.
Indeed, it reduces the nonlinear loss from $f$ to $f_J$ and compresses
the contribution of the eliminated rows $A_{\bar J}$ into the fixed
vector $A_{\bar J}^{\top}s$.

Note that the affine-replacement view differs from the classical view of sample
screening, which eliminates samples by fixing and removing their
corresponding dual variables; see, e.g., \citet{ogawa2013}. Nevertheless,
affine replacement has been used to develop fast optimization algorithms, see e.g., \citep{johnson2018fast}. %Our contribution is to formulate it as a transformation on $\Gamma$ and thereby establish its formal duality with feature screening.

\begin{example}[Sample screening for the soft-margin SVM]
\label{ex:svm-sample}
Classical sample screening for the soft-margin SVM is the special case
\(s=0\) of the proposed sample screening operator \(S^S_{J,s}\). The
soft-margin SVM problem
\[
\min_{x\in\mathbb R^n}
\frac12\|x\|_2^2
+
C\sum_{i=1}^m
\max\!\left(0,\,1-b_i\langle a_i,x\rangle\right),
\qquad C>0,
\]
is represented in FR form by \(p=(f,g,A,0)\), where
\(f(y)=C\sum_{i=1}^m\max(0,1-b_i y_i)\), \(g(x)=\frac12\|x\|_2^2\), and
\(A=(a_1,\ldots,a_m)^\top\).

From \citep{ogawa2013}, the idea of sample screening is that, if one can identify a sample \(i\) that is correctly classified, i.e.,\footnote{The original sample screening rule in~\citep{ogawa2013} also considers the condition
\(
1-b_i\langle a_i,x\rangle>0.
\)
This is another special case covered by our generalized sample screening, but we omit it here for simplicity.}
\begin{equation}
    1-b_i\langle a_i,x\rangle<0
    \label{eq:svm-sample-screening}
\end{equation}
then one can remove the corresponding loss term from the loss function.

Then, index $i$ satisfies the following condition
\begin{equation}
    0\in\partial f_{{i}}(A_{\{i\},:}x).
    \label{eq:svm-sample-screening-subgradient}
\end{equation}
This is exactly the special case \(s=0\) and $\bar J=\{i\}$ of the assumption
\(s\in\partial f_{\bar J}(A_{\bar J}x)\).
Here, note that \eqref{eq:svm-sample-screening} implies \eqref{eq:svm-sample-screening-subgradient}, but not equivalent to it.

In this case, the translation
terms vanish, \(A_{\bar J}^{\top}s=0\) and \(f_{\bar J}^*(s)=0\), so
\(S^S_{J,0}(p)=(f_J,g,A_J,0)\), and the screened problem is
\[
\min_{x\in\mathbb R^n}
\frac12\|x\|_2^2
+
C\sum_{i\in J}
\max\!\left(0,\,1-b_i\langle a_i,x\rangle\right).
\]
The proposed definition thus recovers classical SVM sample screening,
while extending it from samples certified to satisfy \(0\in\partial
f_i\) to samples certified to satisfy an arbitrary subgradient
condition \(s\in\partial f_i\).
\end{example}

\begin{corollary}[Screening reduces problem size]
\label{cor:complexity}
For every \(p=(f,g,A,a)\in\Gamma_{m,n}\), every
\(I\sqcup\bar I=\{1,\ldots,n\}\), and every \(J\sqcup\bar J=\{1,\ldots,m\}\),
\[
\operatorname{size}\bigl(S^F_{I,z}(p)\bigr)=(m,|I|),
\qquad
\operatorname{size}\bigl(S^S_{J,s}(p)\bigr)=(|J|,n).
\]
\end{corollary}

\begin{proof}
Immediate from the definitions above: \(S^F_{I,z}(p)\) has matrix slot
\(A_I\in\mathbb R^{m\times|I|}\), and \(S^S_{J,s}(p)\) has matrix slot
\(A_J\in\mathbb R^{|J|\times n}\).
\end{proof}
\noteCLAUDE{Doc2 dependencies: none. Uses only the definitions of
\(S^F_{I,z}\), \(S^S_{J,s}\) from this section.}

\subsection{Feature--Sample Translation}
In Subsection~\ref{subsec:translation}, we introduced the translation \(T_k\) and its translation parameter \(k\). In this subsection, we introduce feature and sample translation, built on top of \(T_k\) and \(k\).
Surprisingly, these objects also admit equivariance under FR duality.
These notions play an important role in establishing the main results proved in the following subsections.

\begin{definition}[Feature--sample translation]
    \label{def:feature-sample-translation}
    We define the feature translation
    \begin{equation*}
        T^F_{I, z}(p) = T_k(p), \qquad
        k = \kappa^F_{I, z}(p) :=  (A_{\bar I}z,\;g_{\bar I}(z),\;0)\in K_{m, n},
    \end{equation*}
    where \(k\) is a feature translation parameter depending on \(p\), index set \(I\) and \(z\in \mathbb R^{\bar I}\).

    Symmetrically, we define the sample translation
    \begin{equation*}
        T^S_{J, s}(p) = T_h(p), \qquad
        h = \kappa^S_{J, s}(p) := (0,\;-f_{\bar J}^{*}(s),\;A_{\bar J}^{\top}s)\in K_{m, n},
    \end{equation*}
    where \(h\) is a sample translation parameter depending on \(p\), index set \(J\) and \(s\in \mathbb R^{\bar J}\).
\end{definition}

Now we show that there is an equivariance structure on both the
parameter level and the translation level. Definition~\ref{def:equivariance}
(Section~\ref{sec:operators}) defines equivariance for transformations
of \(\Gamma\); the parameter maps \(\kappa^F_{I,z},\kappa^S_{J,s}:\Gamma\to K\)
satisfy the same pattern, now paired with the dual-parameter map
\(k\mapsto k^*\) on \(K\) rather than FR duality on \(\Gamma\).

\begin{proposition}[Parameter equivariance]
\label{prop:parameter-equivariance}
For every \(p=(f,g,A,a)\in\Gamma_{m,n}\), every
\(I\sqcup\bar I=\{1,\ldots,n\}\), \(z\in\mathbb R^{\bar I}\), every
\(J\sqcup\bar J=\{1,\ldots,m\}\), and \(s\in\mathbb R^{\bar J}\),
\begin{align*}
        (\kappa^F_{I, z}(p))^* = \kappa^S_{I, z}(p^*),\\
        (\kappa^S_{J, s}(p))^* = \kappa^F_{J, s}(p^*).
\end{align*}
\end{proposition}

\begin{proof}
Using \(k^*=(-d,-c,-b)\) for \(k=(b,c,d)\) (Section~\ref{sec:operators}),
\[
(\kappa^F_{I,z}(p))^*
=
(A_{\bar I}z,\,g_{\bar I}(z),\,0)^*
=
(0,\,-g_{\bar I}(z),\,-A_{\bar I}z).
\]
Write \(p^*=(g^*,f^*,-A^\top,-a)\). By definition, the sample
translation parameter of \(p^*\) at \((I,z)\) is
\[
\kappa^S_{I,z}(p^*)
=
\bigl(0,\;-(g^*)_{\bar I}^{*}(z),\;(-A^\top)_{\bar I}^{\top}z\bigr).
\]
Since \(g\) is separable, its conjugate restricts termwise,
\((g^*)_{\bar I}=(g_{\bar I})^*\), so
\((g^*)_{\bar I}^{*}(z)=(g_{\bar I})^{**}(z)=g_{\bar I}(z)\) by
biconjugation, using that \(g_{\bar I}\), a finite sum of closed,
proper, convex functions, is itself closed, proper, convex. Rows
\(\bar I\) of \(-A^\top\) equal \(-(A_{\bar I})^\top\), so
\((-A^\top)_{\bar I}^\top z=-A_{\bar I}z\). Hence
\[
\kappa^S_{I,z}(p^*)=(0,\,-g_{\bar I}(z),\,-A_{\bar I}z)=(\kappa^F_{I,z}(p))^*.
\]

We now prove the second identity, which does not follow from the first
by simply exchanging \(f\) and \(g\): Definition~\ref{def:feature-sample-translation}
already builds \(\kappa^F_{I,z}\) and \(\kappa^S_{J,s}\) asymmetrically,
so it needs its own computation.
Using \(k^*=(-d,-c,-b)\) again,
\[
(\kappa^S_{J,s}(p))^*
=
(0,\,-f_{\bar J}^{*}(s),\,A_{\bar J}^{\top}s)^*
=
(-A_{\bar J}^{\top}s,\,f_{\bar J}^{*}(s),\,0).
\]
Write \(p^*=(g^*,f^*,-A^\top,-a)\in\Gamma_{n,m}\). Since \(J\subseteq
\{1,\ldots,m\}\) indexes the second slot's problem size for \(p^*\), the
feature translation parameter of \(p^*\) at \((J,s)\) is
\[
\kappa^F_{J,s}(p^*)
=
\bigl((-A^\top)_{\bar J}s,\;(f^*)_{\bar J}(s),\;0\bigr),
\]
by the same definition as \(\kappa^F_{I,z}\), with \(g^*\) taking the
role of the first slot's function and \(-A^\top\) the role of the
matrix. Columns \(\bar J\) of \(-A^\top\) equal \(-(A_{\bar J})^\top\),
since columns of \(A^\top\) are rows of \(A\) transposed, so
\((-A^\top)_{\bar J}s=-A_{\bar J}^\top s\). Since \(f\) is separable, its
conjugate restricts termwise, \((f^*)_{\bar J}=(f_{\bar J})^*\), so
\((f^*)_{\bar J}(s)=f_{\bar J}^{*}(s)\), the same separability fact used
above, now applied to \(f\).
Hence
\[
\kappa^F_{J,s}(p^*)=(-A_{\bar J}^{\top}s,\,f_{\bar J}^{*}(s),\,0)=(\kappa^S_{J,s}(p))^*.
\]
Unlike the first identity, this computation never invokes
biconjugation: \(\kappa^S_{J,s}(p)\) already carries a conjugate,
\(f_{\bar J}^*(s)\), in its second slot, so dualizing and restricting
commute with a single separability step, not two.
\end{proof}

Proposition~\ref{prop:parameter-equivariance} shows that the
feature--sample asymmetry, a domain shift and a value shift versus a
slope shift and a conjugate value shift, is fully absorbed by the
canonical translation parameters: FR duality exchanges one for the
other. The next step is to lift this parameter equivariance from
parameters to translation operators.

\begin{proposition}[Feature--sample translation equivariance]
\label{prop:translation-equivariance-screening}
For every \(p=(f,g,A,a)\in\Gamma_{m,n}\), every
\(I\sqcup\bar I=\{1,\ldots,n\}\), \(z\in\mathbb R^{\bar I}\), every
\(J\sqcup\bar J=\{1,\ldots,m\}\), and \(s\in\mathbb R^{\bar J}\),
\begin{align*}
        (T^F_{I, z}(p))^* = T^S_{I, z}(p^*),\\
        (T^S_{J, s}(p))^* = T^F_{J, s}(p^*).
\end{align*}
\end{proposition}

\begin{proof}
By Translation equivariance (Proposition~\ref{thm:fenchel-equivariance},
Section~\ref{sec:operators}) applied to \(k=\kappa^F_{I,z}(p)\),
\[
(T^F_{I,z}(p))^*
=
\bigl(T_{\kappa^F_{I,z}(p)}(p)\bigr)^*
=
T_{(\kappa^F_{I,z}(p))^*}(p^*).
\]
By Parameter equivariance (Proposition~\ref{prop:parameter-equivariance}),
\((\kappa^F_{I,z}(p))^*=\kappa^S_{I,z}(p^*)\), so
\[
(T^F_{I,z}(p))^*
=
T_{\kappa^S_{I,z}(p^*)}(p^*)
=
T^S_{I,z}(p^*).
\]
The second identity follows the same argument, exchanging the roles of
\(f\) and \(g\), \(I\) and \(J\), and domain and slope translation.
\end{proof}

\subsection{Decomposition of Screening}

This subsection proves that feature screening and sample screening are
not primitive: each decomposes into a translation followed by a
restriction.

\begin{theorem}[Screening decomposition]
\label{thm:decomposition}
For every \(p=(f,g,A,a)\), every \(I\sqcup\bar I=\{1,\ldots,n\}\),
\(z\in\mathbb R^{\bar I}\), \(J\sqcup\bar J=\{1,\ldots,m\}\), and
\(s\in\mathbb R^{\bar J}\),
\[
S^{F}_{I,z}
=
R^{F}_{I}\circ T^{F}_{I,z},
\qquad
S^{S}_{J,s}
=
R^{S}_{J}\circ T^{S}_{J,s}.
\]
\end{theorem}

\begin{proof}
By Definition~\ref{def:feature-sample-translation}, \(\kappa^F_{I,z}(p)\) has zero slope component, so
\[
T^F_{I,z}(p)
=
\bigl(T^{\mathrm d}_{A_{\bar I}z}f,\;g,\;A,\;a+g_{\bar I}(z)\bigr).
\]
Applying \(R^F_I\), which restricts the second and third slots to
\(I\) and leaves the first and fourth untouched,
\[
R^F_I\bigl(T^F_{I,z}(p)\bigr)
=
\bigl(T^{\mathrm d}_{A_{\bar I}z}f,\;g_I,\;A_I,\;a+g_{\bar I}(z)\bigr),
\]
which is exactly \(S^F_{I,z}(p)\) from the definition above.

The sample case follows by an analogous argument. Since \(\kappa^S_{J,s}(p)\)
has zero domain component,
\[
T^S_{J,s}(p)
=
\bigl(f,\;T^{\mathrm s}_{A_{\bar J}^{\top}s}g,\;A,\;a-f_{\bar J}^{*}(s)\bigr).
\]
Applying \(R^S_J\), which restricts the first and third slots to \(J\)
and leaves the second and fourth untouched,
\[
R^S_J\bigl(T^S_{J,s}(p)\bigr)
=
\bigl(f_J,\;T^{\mathrm s}_{A_{\bar J}^{\top}s}g,\;A_J,\;a-f_{\bar J}^{*}(s)\bigr),
\]
which is exactly \(S^S_{J,s}(p)\).
\end{proof}
\noteCLAUDE{Doc2 dependencies: Definition "Translation" (\(T_k(p)\))
and Definition "Feature and sample restriction" (\(R^F_I\), \(R^S_J\))
only. No proposition or lemma needed.} %\noteLE{Verified.}

Although feature and sample screening arise from different assumptions, a fixed primal block and a fixed subgradient block, they share the same two-step construction: a primitive translation followed by a primitive restriction. This common structure, together with the equivariance of translation and restriction under FR duality, provides the foundation for establishing the duality between feature and sample screening.

\subsection{Duality of Screening}

This subsection proves that feature screening and sample screening are
equivariant under FR duality: dualizing one produces exactly the other,
on the dual representation. Feature screening moves the eliminated
block of \(x\) into a domain translation of \(f\), then discards it
from \(g\). FR duality exchanges domain translations
with slope translations and exchanges the roles of \(f\) and \(g\)
(Section~\ref{sec:operators}). Sample screening moves an eliminated
block of \(y\) into a slope translation of \(g\), then discards it
from \(f\). Matching these two descriptions suggests that dualizing a
feature screening produces exactly a sample screening, with no other
operation involved.

\begin{theorem}[Feature--sample duality]
\label{thm:duality}
Let \(p=(f,g,A,a)\in\Gamma_{m,n}\). For every
\(I\sqcup\bar I=\{1,\ldots,n\}\) and \(z\in\mathbb R^{\bar I}\),
\(S^F_{I,z}\) and \(S^S_{I,z}\) are equivariant under FR duality
(Definition~\ref{def:equivariance}, Section~\ref{sec:operators}):
\[
\left(S^{F}_{I,z}(p)\right)^*
=
S^{S}_{I,z}(p^*).
\]
Symmetrically, for every \(J\sqcup\bar J=\{1,\ldots,m\}\) and
\(s\in\mathbb R^{\bar J}\),
\[
\left(S^{S}_{J,s}(p)\right)^*
=
S^{F}_{J,s}(p^*).
\]
\end{theorem}

\begin{proof}
We prove the first identity; the second follows by a symmetric argument
with the roles of \(f\) and \(g\), \(I\) and \(J\), domain and slope
translation, exchanged. The proof combines Screening decomposition
(Theorem~\ref{thm:decomposition}, this section) with Restriction
equivariance (Section~\ref{sec:operators}) and Feature--sample
translation equivariance
(Proposition~\ref{prop:translation-equivariance-screening}, this
section).

By Theorem~\ref{thm:decomposition}, \(S^F_{I,z}(p)=R^F_I(T^F_{I,z}(p))\).
By Restriction equivariance, applied to \(q:=T^F_{I,z}(p)\),
\[
\left(S^F_{I,z}(p)\right)^*
=
\left(R^F_I(q)\right)^*
=
R^S_I(q^*)
=
R^S_I\bigl((T^F_{I,z}(p))^*\bigr).
\]
By Feature--sample translation equivariance,
\((T^F_{I,z}(p))^*=T^S_{I,z}(p^*)\). Substituting,
\[
\left(S^F_{I,z}(p)\right)^*
=
R^S_I\bigl(T^S_{I,z}(p^*)\bigr)
=
S^S_{I,z}(p^*),
\]
the last equality by Theorem~\ref{thm:decomposition} applied to \(p^*\).
This proves the first identity.

The second identity is the mirror computation, exchanging the roles of
\(f\) and \(g\), \(I\) and \(J\), and domain and slope translation,
using the second halves of Screening decomposition, Restriction
equivariance, and Feature--sample translation equivariance.
\end{proof}
\noteCLAUDE{Dependencies: Theorem~\ref{thm:decomposition} and
Proposition~\ref{prop:translation-equivariance-screening} (this
section), plus Restriction equivariance (doc2). No convex-analysis fact
is needed inline anymore: the separability/biconjugation argument now
lives entirely inside Proposition~\ref{prop:parameter-equivariance}'s
proof, cited transitively through
Proposition~\ref{prop:translation-equivariance-screening}.}

Theorem~\ref{thm:duality} says a single screening dualizes to a single
screening of the other type. What happens when a feature screening and a sample screening are both applied to the same \(p\)? 
This composition was used in \citep{shibagaki2016simultaneous} during the solving process, where it is referred to as simultaneous screening. The following result shows that, by exchanging their order with the appropriate parameters, the two compositions also form an equivariant pair.

\begin{corollary}[Simultaneous screening equivariance]
\label{cor:simultaneous}
For every \(p=(f,g,A,a)\in\Gamma_{m,n}\), every
\(I\sqcup\bar I=\{1,\ldots,n\}\), \(z\in\mathbb R^{\bar I}\),
\(J\sqcup\bar J=\{1,\ldots,m\}\), and \(s\in\mathbb R^{\bar J}\),
whenever both compositions below are defined,
\[
\left(S^{S}_{J,s}\,S^{F}_{I,z}(p)\right)^*
=
S^{F}_{J,s}\,S^{S}_{I,z}(p^*),
\]
equivalently,
\[
S^{S}_{J,s}\,S^{F}_{I,z}(p)
=
\left(S^{F}_{J,s}\,S^{S}_{I,z}(p^*)\right)^*.
\]
\end{corollary}

\begin{proof}
Write \(q:=S^F_{I,z}(p)\). Applying Theorem~\ref{thm:duality} to \(q\) and \(p\), respectively, we obtain
\[
\left(S^S_{J,s}(q)\right)^*=S^F_{J,s}(q^*),
\qquad 
q^*=\left(S^F_{I,z}(p)\right)^*=S^S_{I,z}(p^*).
\]
Substituting,
\[
\left(S^S_{J,s}\,S^F_{I,z}(p)\right)^*
=
S^F_{J,s}\bigl(S^S_{I,z}(p^*)\bigr)
=
S^F_{J,s}\,S^S_{I,z}(p^*),
\]
which is the first identity. The second follows by applying
\((\cdot)^*\) to both sides and using the involution \((p^*)^*=p\)
(Proposition~\ref{prop:involution}).
\end{proof}
% \noteCLAUDE{Doc2 dependencies: only Proposition "Involution"
% (\((p^*)^*=p\)). Everything else comes from Theorem~\ref{thm:duality}
% (this section), applied twice.}\noteLE{Verified}

% This is the correct replacement for a Commutativity theorem. Feature
% screening and sample screening do not need to commute on the same
% representation \(p\): requiring that would identify two representations
% that, in general, differ by exactly the cross-block interaction
% between the dropped samples and the dropped features. Instead,
% composing on \(p\) and dualizing lands on the same object as composing
% on \(p^*\), with \(S^F\) and \(S^S\) exchanging roles. This is what
% "simultaneous screening" means in this calculus: not that the two
% operators commute, but that composing them is itself compatible with
% duality, exactly as a single screening is (Theorem~\ref{thm:duality}).
% This is a two-line corollary of Theorem~\ref{thm:duality}, applied
% twice; no new machinery is needed beyond it.

\section*{Conclusion}

This paper formalizes the duality between feature and sample screening. Our first contribution is to introduce Fenchel--Rockafellar (FR) representations, a class of convex problems on which screening and FR duality act as closed transformations. Our second contribution is to prove that feature and sample screening are equivariant under FR duality: feature screening followed by dualization is equivalent to dualization followed by sample screening, and conversely. Their duality is therefore established as a commutative relation between transformations rather than an informal, model-dependent correspondence.

The proof rests on two primitive transformations, translation and restriction, and the derived notions of feature and sample translation. We show that these transformations are equivariant under FR duality and that every feature or sample screening operator decomposes into a translation followed by a restriction. The equivariance of feature and sample screening then follows directly from these two structural results.

The framework has two main limitations. First, restriction relies on separability of the FR representation, capturing screening methods beyond this setting \citep{elvira2020safe,nguyen2026gap} requires a broader framework. Second, the present framework characterizes the algebraic duality of general screening, but not safe screening. Investigating the change of solution sets after screening is also an important direction toward a complete duality theory of safe screening.

% \pagebreak
\bibliographystyle{plainnat}
\bibliography{meta/refs}

\end{document}